\documentclass[12pt]{article}

\RequirePackage[OT1]{fontenc}
\RequirePackage{amsthm,amsmath}
\RequirePackage[colorlinks,citecolor=blue,urlcolor=blue,linkcolor=blue]{hyperref}

\usepackage{amsmath}
\usepackage{amssymb}
\usepackage{amsthm}
\usepackage{dsfont}
\usepackage{graphicx}
\usepackage{color}
\usepackage{tikz}
\usepackage[margin=3cm]{geometry}

\definecolor{mycolor1}{rgb}{1,0,0}
\definecolor{mycolor2}{rgb}{0,0,1}
\definecolor{mycolor3}{rgb}{0.169,0.506,0.337}

\numberwithin{equation}{section}
\theoremstyle{plain}
\newtheorem{theorem}{Theorem}[section]

\newtheorem{definition}[theorem]{Definition}
\newtheorem{lemma}[theorem]{Lemma}
\newtheorem{proposition}[theorem]{Proposition}

\newtheorem{remark}[theorem]{Remark}

\newcommand{\RR}{\mathbb{R}}
\newcommand{\ZZ}{\mathbb{Z}}
\newcommand{\CC}{\mathbb{C}}
\newcommand{\EE}{\mathbb{E}}

\newcommand{\ddiag}{\operatorname{ddiag}}
\newcommand{\trace}[1]{\operatorname{Tr}\left(#1\right)}
\newcommand{\diag}{\operatorname{diag}}
\newcommand{\rank}{\operatorname{rank}}
\newcommand{\polylog}{\operatorname{polylog}}

\newcommand{\opnorm}[1]{\left\|#1\right\|_{\mathrm{op}}}

\newcommand{\inner}[2]{\left\langle #1, #2 \right\rangle }

\begin{document}

\title{Tightness of the maximum likelihood semidefinite relaxation for angular synchronization}

\author{Afonso S.\ Bandeira\thanks{Massachusetts Institute
		of Technology. bandeira@mit.edu, \url{http://math.mit.edu/\textasciitilde bandeira/}.}, \qquad Nicolas Boumal\thanks{Princeton University. nboumal@princeton.edu, \url{http://www.math.princeton.edu/\textasciitilde nboumal/}.}, \qquad Amit Singer\thanks{Princeton University. amits@math.princeton.edu, \url{http://www.math.princeton.edu/\textasciitilde amits/}.}}

\maketitle

{\let\thefootnote\relax\footnote{Authors are listed alphabetically. Please direct correspondence to nboumal@princeton.edu.}}

\begin{abstract}
	
Maximum likelihood estimation problems are, in general, intractable optimization problems. As a result, it is common to approximate the maximum likelihood estimator (MLE) using convex relaxations. In some cases, the relaxation is tight: it recovers the true MLE. Most tightness proofs only apply to situations where the MLE exactly recovers a planted solution (known to the analyst).
It is then sufficient to establish that the optimality conditions hold at the planted signal.
In this paper, we study an estimation problem (angular synchronization) for which the MLE is not a simple function of the planted solution, yet for which the convex relaxation is tight. To establish tightness in this context, the proof is less direct because the point at which to verify optimality conditions is not known explicitly.

Angular synchronization consists in estimating a collection of $n$ phases, given noisy measurements of the pairwise relative phases.
The MLE for angular synchronization is the solution of a (hard) non-bipartite Grothendieck problem over the complex numbers.
We consider a stochastic model for the data: a planted signal (that is, a ground truth set of phases) is corrupted with non-adversarial random noise.
Even though the MLE does not coincide with the planted signal, we show that the classical semidefinite relaxation for it is tight, with high probability. This holds even for high levels of noise.

\vspace{3mm}
Keywords: 
{Angular Synchronization};
{Semidefinite programming};
{Tightness of convex relaxation};
{Maximum likelihood estimation}.

\end{abstract}

\clearpage
\section{Introduction}

Recovery problems in statistics and many other fields are commonly solved under the paradigm of maximum likelihood estimation, partly due to the rich theory it enjoys. Unfortunately, in many important applications, the parameter space is exponentially large and non-convex, often rendering the computation of the maximum likelihood estimator (MLE) intractable. It is then common to settle for heuristics, such as expectation-maximization algorithms to name but one example. However, it is also common for such iterative heuristics to get trapped in local optima. Furthermore, even when these methods do attain a global optimum, there is in general no way to verify this.

A now classic alternative to these heuristics is the use of convex relaxations. The idea is to maximize the likelihood in a larger, convex set that contains the parameter space of interest, as (well-behaved) convex optimization problems are generally well understood and can be solved in polynomial time. The downside is that the solution obtained might not be in the original feasible (acceptable) set. One is then forced to take an extra, potentially suboptimal, rounding step.

This line of thought is the basis for a wealth of modern \emph{approximation algorithms}~\cite{williamson2011design}. One preeminent example is Goemans and Williamson's treatment of Max-Cut~\cite{goemans1995maxcut}, an NP-hard combinatorial problem that involves segmenting a graph in two clusters to maximize the number of edges connecting the two clusters. They first show that Max-Cut can be formulated as a semidefinite program (SDP)---a convex optimization problem where the variable is a positive semidefinite matrix---with the additional, non-convex constraint that the sought matrix be of rank one. Then, they propose to solve this SDP while relaxing (removing) the rank constraint. They show that the obtained solution, despite typically being of rank strictly larger than one, can be rounded to a (suboptimal) rank-one solution, and that it provides a guaranteed approximation of the optimal value of the hard problem. Results of the same nature abound in the recent theoretical computer science literature~\cite{goemans2004approximation,zhang2006complex,briet2014grothendieck,demanet2013convex,Bandeira_Singer_Spielman_OdCheeger,Bandeira_LittleGrothendieckOd,so2007approximating,luo2010semidefinite,zhang2006complex}.

In essence, approximation algorithms insist on solving \emph{all} instances of a given NP-hard problem in polynomial time, which, unless P = NP, must come at the price of accepting some degree of sub-optimality. This worst-case approach hinges on the fact that a problem is NP-hard as soon as every efficient algorithm for it can be hindered by at least one pathological instance.

Alternatively, in a non-adversarial setting where ``the data is not an enemy,'' one may find that such pathological cases are not prominent. As a result, the applied mathematics community has been more interested in identifying regimes for which the convex relaxations are tight, that is, admit a solution that is also admissible for the hard problem. When this is the case, no rounding is necessary and a truly optimal solution is found in reasonable time, together with a certificate of optimality.
This is sometimes achieved by positing a probability distribution on the instances of the problem and asserting tightness \emph{with high probability}, as for example in compressed sensing~\cite{Candes_CS1,Donoho_CS,Tropp_justrelax}, matrix completion~\cite{candes2009exact},
clustering%
~\cite{Abbe_Z2Synch} and inverse problems~\cite{Venkat_Geometryofinverseproblems,Tropp:LivingOnTheEdge}. In this approach, one surrenders the hope to solve all instances of the hard problem, in exchange for true optimality with high probability.

The results presented in this work are of that nature. As a concrete object of study, we consider the \emph{angular synchronization problem}~\cite{ASinger_2011_angsync,Bandeira_Singer_Spielman_OdCheeger}, which consists in estimating a collection of $n$ phases $e^{i\theta_1}, \ldots, e^{i\theta_n}$, given noisy measurements of pairwise relative phases $e^{i(\theta_k - \theta_\ell)}$---see Section~\ref{sec:synchro} for a formal description. This problem notably comes up in time-synchronization of distributed networks~\cite{giridhar2006distributed}, signal reconstruction from phaseless measurements~\cite{Alexeev_PhaseRetrievalPolarization,Bandeira_FourierMasks}, ranking~\cite{cucuringu2015syncrank}, digital communications~\cite{So_2010_SDP_detector}, and surface reconstruction problems in computer vision~\cite{agrawal2006surface} and optics~\cite{rubinstein2001optical}. Angular synchronization serves as a model for the more general problem of synchronization of rotations in any dimension, which comes up in structure from motion~\cite{matinec2007rotation,hartley2013rotation}, surface reconstruction from 3D scans~\cite{Wang_RobustSynchronization} and cryo-electron microscopy~\cite{ASinger_YShkolnisky_commonlines}, to name a few.

The main contribution of the present paper is a proof that, even though the angular synchronization problem is NP-hard~\cite{zhang2006complex}, its MLE in the face of Gaussian noise can often be computed (and certified) in polynomial time. This remains true even for entry-wise noise levels growing to infinity as the size of the problem (the number of phases) grows to infinity. The MLE is obtained as the solution of a semidefinite relaxation described in Section~\ref{sec:synchro}. This arguably striking phenomenon has been observed empirically before~\cite{bandeira2014open} (see also Figure~\ref{fig:experiment}), but not explained.

Computing the MLE for angular synchronization is equivalent to solving a non-bipartite Grothendieck problem. Semidefinite relaxations for Gro\-then\-dieck problems have been thoroughly studied in theoretical computer science from the point of view of approximation ratios. The name is inspired by its close relation to an inequality of Grothendieck~\cite{Grothendieck_GT}. We direct readers to the survey by Pisier~\cite{Pisier_GT} for a discussion. 

The proposed result is qualitatively different from most tightness results available in the literature. Typical results establish either exact recovery of a planted signal~\cite{Abbe_Z2Synch,huang2013consistent,Ames_14_Clustering} (mostly in discrete settings), or exact recovery in the absence of noise, joint with stable (but not necessarily optimal) recovery when noise is present~\cite{Candes_Strohmer_Voroninski_phaselift,Demanet_Phaseless,Candes_CS2,Wang_RobustSynchronization,demanet2013convex,zhang2015disentangling}. In contrast, this paper shows optimal recovery even though exact recovery is not possible. In particular, Demanet and Jugnon showed stable recovery for angular synchronization via semidefinite programming, under adversarial noise~\cite{demanet2013convex}. We complement this by showing tightness in a non-adversarial setting, meaning the actual MLE is computed.

A similar semidefinite relaxation was studied in a digital communications context, where the parameters to estimate are $m$th roots of unity~\cite{So_2010_SDP_detector}. There, non-asymptotic results show the relaxation approximates the MLE within some factor, with high probability. The present paper is related to the limit $m\to\infty$, although the noise model considered is different, and we focus on exact MLE computation.

Our proof relies on verifying that a certain candidate dual certificate is valid with high probability. The main difficulty comes from the fact that the dual certificate depends on the MLE, which does not coincide with the planted signal, and is a nontrivial function of the noise. We use necessary optimality conditions of the hard problem to both obtain an explicit expression for the candidate dual certificate, and to partly characterize the point whose optimality we aim to establish. This seems to be required since the MLE is not known in closed form.

In the context of sparse recovery, a result with similar flavor is support recovery guarantee~\cite{Tropp_justrelax}, where the support of the estimated signal is shown to be contained in the support of the original signal. Due to the noise, exact recovery is also impossible in this setting. Another example is a recovery guarantee in the context of latent variable selection in graphical models~\cite{Venkat_GraphicalModels}.

Besides the relevance of angular synchronization in and of its own, we are confident this new insight will help uncover similar results in other applications where it has been observed that semidefinite relaxations can be tight even when the ground truth cannot be recovered. Notably, this appears to be the case for the Procrustes and multi-reference alignment problems~\cite{Bandeira_Charikar_Singer_Zhu_Alignment,bandeira2014open}).

The crux of our argument concerns the rank of the solutions of an SDP.
We mention in passing that there are many other deterministic results in the literature pertaining to the rank of solutions of SDP's. For example, it has been shown %
that, in general, an SDP with only equality constraints admits a solution of rank at most (on the order of) the square root of the number of constraints, see~\cite{Shapiro_82_rank,pataki1998rank,barvinok1995problems}. Furthermore, Sagnol~\cite{sagnol2011rankone} shows that under some conditions (that are not fulfilled in our case), certain SDP's related to packing problems always admit a rank-one solution. Sojoudi and Lavaei~\cite{sojoudi2014exactness} study a class of SDP's on graphs which is related to ours and for which, under certain strong conditions on the topology of the graphs, the SDP's admit rank-one solutions---see also applications to power flow optimization~\cite{low2013convex}.

\subsection{Contribution}

We frame angular synchronization as the problem of estimating $z\in\CC^n$ with $|z_i| = 1 \ \forall i$, given measurements $C=zz^*+\sigma W$, where $W = W^*$ has i.i.d.\ complex standard Gaussian entries above its diagonal. The MLE $x$ maximizes $x^*Cx$ over the parameter space~\eqref{eq:P}. This is relaxed to $\max \trace{CX}$ s.t.\ $\diag(X) = \mathds{1}, X~\succeq~0$~\eqref{eq:SDP}. After proving exact recovery for \emph{real} $x$ and $X$ (Fig.~\ref{fig:realrecovery}), we focus on the \emph{complex} case (Fig.~\ref{fig:experiment}). Our main result is that, if $\sigma \leq \frac{1}{18}n^{1/4}$, then, with high probability, \eqref{eq:SDP} admits a unique solution of rank 1, revealing the global optimum of~\eqref{eq:P} (Thm.~\ref{Thm:rr_gaussian}). In order to do so, we characterize global optimizers of~\eqref{eq:P} as points $x$ which satisfy first- and second-order necessary optimality conditions and perform at least as well as $z$ in terms of the cost function of~\eqref{eq:P} (under conditions on the perturbation $\sigma W$).


\subsection{Notation}

For $a \in \CC$, $\overline a$ denotes its complex conjugate and $|a| = \sqrt{a\overline a}$ its modulus. For $v \in \CC^n$, $v^*$ denotes its conjugate transpose, $\|v\|^2 = \|v\|_2^2 = v^*v$, $\|v\|_\infty = \max_i |v_i|$ and $\|v\|_1 = \sum_i |v_i|$. $\mathds{1}$ is a vector of all-ones. $\diag(v)$ is a diagonal matrix with $\diag(v)_{ii} = v_i$. For a matrix $M$, $\opnorm{M}$ is the maximal singular value and
$\Re(M)$, $\Im(M)$
extract
the real
and imaginary
parts; $\diag(M)$ extracts the diagonal of $M$ into a vector and $\ddiag(M)$ sets all off-diagonal entries of $M$ to zero. We use the real inner product $\inner{u}{v} = \Re\{u^*v\}$ over $\CC^n$.
$\EE$ denotes mathematical expectation.

\section{The Angular Synchronization problem}
\label{sec:synchro}

We focus on the problem of \emph{angular synchronization}~\cite{ASinger_2011_angsync,Bandeira_Singer_Spielman_OdCheeger}, in which one wishes to estimate a collection of $n$ phases ($n\geq 2$) based on measurements of pairwise phase differences. We restrict our analysis to the case where a measurement is available for every pair of phases. More precisely, we let $z \in \CC^n$ be an unknown, complex vector with unit modulus entries, $|z_1| = \cdots = |z_n| = 1$, and we consider measurements of the form $C_{ij} = z_i\overline{z_j} + \varepsilon_{ij}$, where
$\varepsilon_{ij} \in \CC$ is noise affecting the measurement. By symmetry, we define $C_{ji} = \overline{C_{ij}}$ and $C_{ii} = 1$, so that the matrix $C \in \CC^{n \times n}$ whose entries are given by the $C_{ij}$'s is Hermitian.

Further letting the noise $\varepsilon_{ij}$ be i.i.d.\ (complex) Gaussian variables for $i<j$, it follows that an MLE for $z$ is any vector of phases $x$ minimizing $\sum_{i,j} |C_{ij}x_j - x_i|^2$. Equivalently, an MLE is a solution of the following quadratically constrained quadratic program (sometimes called the complex constant-modulus QP~\cite[Table\,2]{luo2010semidefinite} in the optimization literature, and non-bipartite Grothendieck problem in theoretical computer science):
\begin{align}
	\max_{x \in \CC^n} \ x^* C x, \ \textrm{ subject to } |x_1| = \cdots = |x_n| = 1,
	\tag{QP}
	\label{eq:P}
\end{align}
where $x^*$ denotes the conjugate transpose of $x$. This problem can only be solved up to a global phase, since only relative information is available. Indeed, given any solution $x$, all vectors of the form $x e^{i\theta}$ are equivalent solutions, for arbitrary phase $\theta$. %

Solving~\eqref{eq:P} is, in general, an NP-hard problem~\cite[Prop.\,3.5]{zhang2006complex}. It is thus unlikely that there exists an algorithm capable of solving~\eqref{eq:P} in polynomial time for an arbitrary cost matrix $C$. In response, building upon now classical techniques, various authors~\cite{zhang2006complex,so2007approximating,ASinger_2011_angsync,Abbe_Z2Synch,abbe2014exact,Bandeira_LittleGrothendieckOd} study the following convex relaxation of~\eqref{eq:P}. For any admissible $x$, the Hermitian matrix $X = xx^* \in \CC^{n\times n}$ is Hermitian positive semidefinite, has unit diagonal entries and is of rank one. Conversely, any such $X$ may be written in the form $X = xx^*$ such that $x$ is admissible for~\eqref{eq:P}. In this case, the cost function can be rewritten in linear form: $x^* C x = \trace{x^* C x} = \trace{CX}$. Dropping the rank constraint then yields the relaxation we set out to study:
\begin{align}
	\max_{X \in \CC^{n\times n}} \ \trace{CX}, \ \textrm{ subject to } \diag(X) = \mathds{1} \textmd{ and } X \succeq 0.
	\tag{SDP}
	\label{eq:SDP}
\end{align}
Such relaxations \emph{lift} the problem to higher dimensional spaces. Indeed, the search space of~\eqref{eq:P} has dimension $n$ (or $n-1$, discounting the global phase) whereas the search space of~\eqref{eq:SDP} has real dimension $n(n-1)$. In general, increasing the dimension of an optimization problem may not be advisable. But in this case, the relaxed problem is a semidefinite program. Such optimization problems can be solved to global optimality up to arbitrary precision in polynomial time~\cite{LVanderberghe_SBoyd_1996}.

It is known that the solution of~\eqref{eq:SDP} can be rounded to an approximate solution of~\eqref{eq:P}, with a guaranteed approximation ratio~\cite[\S\,4]{so2007approximating}. But even better, when~\eqref{eq:SDP} admits an optimal solution $X$ of rank one, then no rounding is necessary: the leading eigenvector $x$ of $X = xx^*$ is a global optimum of~\eqref{eq:P}, meaning we have solved the original problem exactly. Elucidating when the semidefinite program admits a solution of rank one, i.e., when the relaxation is \emph{tight}, is the focus of the present paper.

\subsection{A detour through synchronization over $\ZZ_2 = \{ \pm 1 \}$}

Problem~\eqref{eq:P} is posed over the complex numbers. As a result, the individual variables $x_i$ in~\eqref{eq:P} live on a continuous search space (the unit circle). One effect of this is that even small noise on the data precludes exact recovery of the signal $z$ (in general). This is the root of most of the complications that will arise in the developments hereafter. In order to first illustrate some of the ideas of the proof in a simpler context, this section proposes to take a detour through the real case. Besides this expository rationale, the real case is interesting in and of itself. It notably relates to correlation clustering~\cite{Abbe_Z2Synch} and the stochastic block model~\cite{abbe2014exact}. The analysis proposed here also appears in~\cite{bandeira2015laplacian}.

Let $z \in \{\pm 1\}^n$ be the signal to estimate and let $C = zz^{\!\top} + \sigma W$ contain the measurements, with $W = W^{\!\top}$ a Wigner matrix: its above-diagonal entries are i.i.d.\ (real) standard normal random variables, and its diagonal entries are zero. Each entry $C_{ij}$ is a noisy measurement of the relative sign $z_iz_j$. For example, $z_i$ could represent political preference of an agent (left or right wing) and $C_{ij}$ could be a measurement of agreement between two agents' views~\cite{abbe2014exact}. Consider this MLE problem:
\begin{align}
	\max_{x \in \RR^n} \ x^{\!\top} C x, \ \textrm{ subject to } |x_1| = \cdots = |x_n| = 1.
	\tag{QP$\mathbb{R}$}
	\label{eq:P_real}
\end{align}
Thus, $x_i \in \{\pm 1\}$ for each $i$. The corresponding relaxation reads:
\begin{align}
	\max_{X \in \RR^{n\times n}} \ \trace{CX}, \ \textrm{ subject to } \diag(X) = \mathds{1} \textmd{ and } X \succeq 0.
	\tag{SDP$\mathbb{R}$}
	\label{eq:SDP_real}
\end{align}
Certainly, if~\eqref{eq:SDP_real} admits a rank-one solution, then that solution reveals a global optimizer of~\eqref{eq:P_real}. Beyond rank recovery, because the variables $x_i$ are discrete, it is expected that, if noise is sufficiently small, then the global optimizer of~\eqref{eq:P_real} will be the true signal $z$.
Figure~\ref{fig:realrecovery} confirms this even for large noise, leveraging the tightness of the relaxation~\eqref{eq:SDP_real}. Note that exact recovery is a strictly stronger requirement than rank recovery; for example, for $n=2$, the relaxation is always tight~\cite{pataki1998rank}, but it is not always exact. We now investigate this exact recovery phenomenon for~\eqref{eq:SDP_real}.

\begin{figure}[h]
\centering
\begin{tikzpicture}
\draw[draw=none,shape=rectangle]
    (0,0) node {\includegraphics[width=\linewidth]{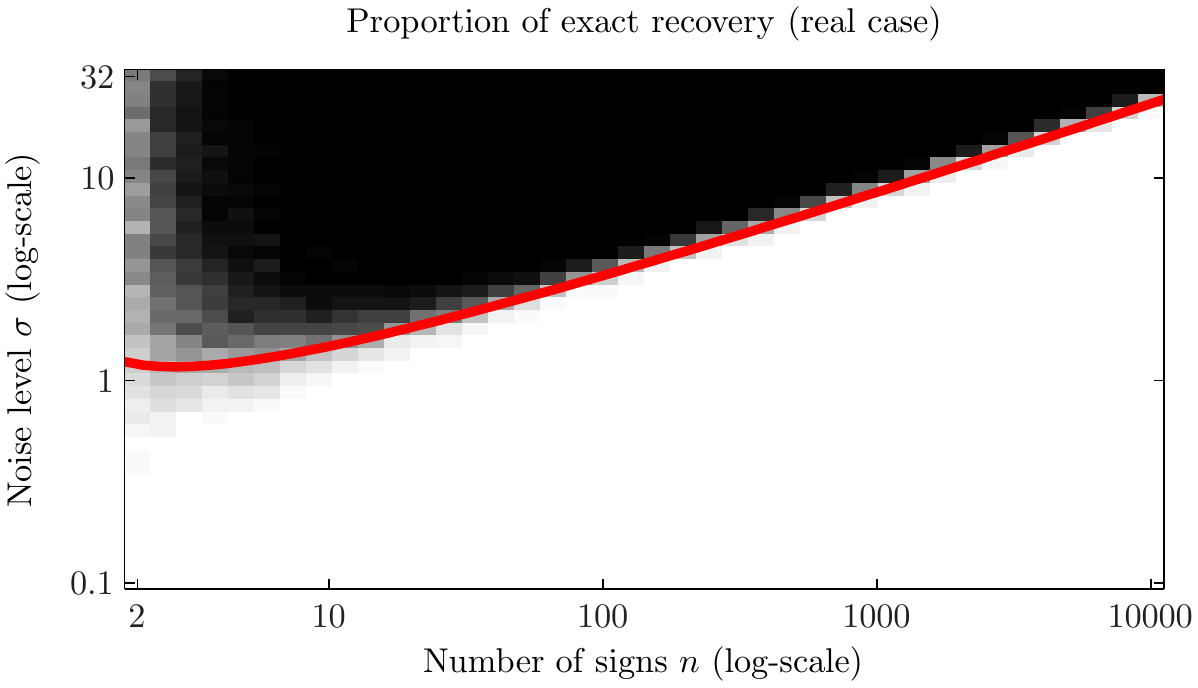}};
\draw[draw=none,shape=rectangle] (3, 1) node[mycolor1] {$\sqrt{\frac{n}{2\log{n}}}$};
\end{tikzpicture}
\vspace{-3mm}
\caption{In the real case~\eqref{eq:P_real}, one can hope to recover the signal $z \in \{\pm 1\}^n$ exactly from the pairwise sign comparisons $zz^{\!\top} + \sigma W$. This figure shows how frequently the semidefinite relaxation~\eqref{eq:SDP_real} returns the correct signal $z$ (or $-z$). For each pair $(n, \sigma)$, 100 realizations of the noise $W$ are generated independently and the dual certificate for the true signal~\eqref{eq:S_real} is verified (it is declared numerically positive semidefinite if its smallest eigenvalue exceeds $-10^{-14}n$). The frequency of success is coded by intensity (bright for 100\% success, dark for 0\% success). The results are in excellent agreement with the theoretical predictions~\eqref{eq:real_rate} (\textcolor{mycolor1}{red curve}).}
\label{fig:realrecovery}
\end{figure}

The aim is to show that $X = zz^{\!\top}$ is a solution of~\eqref{eq:SDP_real}. Semidefinite programs admit a dual problem~\cite{LVanderberghe_SBoyd_1996}, which for~\eqref{eq:SDP_real} reads:
\begin{align}
	\min_{S \in \RR^{n\times n}} \ \trace{S+C}, \textrm{ s.t. } S+C \text{ is diagonal} \textmd{ and } S \succeq 0.
	\tag{DSDP$\mathbb{R}$}
	\label{eq:SDP_dual_real}
\end{align}
Strong duality holds, which implies that a given feasible $X$ is optimal if and only if there exists a dual feasible matrix $S$ such that $\trace{CX} = \trace{S+C}$, or, in other words, such that $\trace{SX} = 0$.\footnote{Indeed, $S+C$ is diagonal and $X_{ii}=1$, hence $\trace{S+C} = \trace{(S+C)X} = \trace{SX} + \trace{CX}$.} Since both $S$ and $X$ are positive semidefinite, this is equivalent to requiring $SX = 0$ (a condition known as \emph{complementary slackness}), and hence requiring $Sz = 0$.

For ease of exposition, we now assume (without loss of generality) that $z = \mathds{1}$. Tentatively, let $S = nI + \sigma \diag(W\mathds{1}) - C$---for the complex case, we will see how to obtain this candidate without guessing. Then, by construction, $S+C$ is diagonal and $S\mathds{1} = 0$.\footnote{Using $C = \mathds{1}\mathds{1}^{\!\top} + \sigma W$, we get $S\mathds{1} = n\mathds{1} + \sigma W\mathds{1} - n\mathds{1} - \sigma W\mathds{1} = 0$.} It remains to determine under what conditions $S$ is positive semidefinite.

Define the diagonal matrix $D_W = \diag(W\mathds{1})$ and the (Laplacian-like) matrix $L_W = D_W - W$. The candidate dual certificate is a sum of two Laplacian-like matrices:
\begin{align}
	S = nI - \mathds{1}\mathds{1}^{\!\top} + \sigma L_W.
	\label{eq:S_real}
\end{align}
The first part is the Laplacian of the complete graph. It has eigenvalues $0, n, \ldots, n$, with the zero eigenvalue corresponding to the all-ones vector $\mathds{1}$. Since $L_W\mathds{1} = 0$ by construction, $S$ is positive semidefinite as long as $\sigma L_W$ does not ``destroy'' any of the large eigenvalues $n$. This is guaranteed in particular if $\opnorm{\sigma L_W} < n$. It is well-known from concentration results about Wigner matrices that, for all $\varepsilon > 0$, with high probability for large $n$, $\opnorm{L_W} \leq \sqrt{(2+\varepsilon)n \log n}$ 
(see~\cite[Thm.\,1]{Ding_RandomLaplacians} and~\cite{bandeira2015laplacian}).
Thus, it is expected that~\eqref{eq:SDP_real} will yield exact recovery of the signal $z$ (with high probability) as long as
\begin{align}
\sigma < \sqrt{\frac{n}{(2+\varepsilon)\log n}}.
\label{eq:real_rate}
\end{align}
This is indeed compatible with the empirical observation of Figure~\ref{fig:realrecovery}.

\subsection{Back to synchronization over SO(2)}
\label{sec:synchroSO2roadmap}

\begin{figure}
\centering
\begin{tikzpicture}
\draw[draw=none,shape=rectangle]
    (0,0) node {\includegraphics[width=\linewidth]{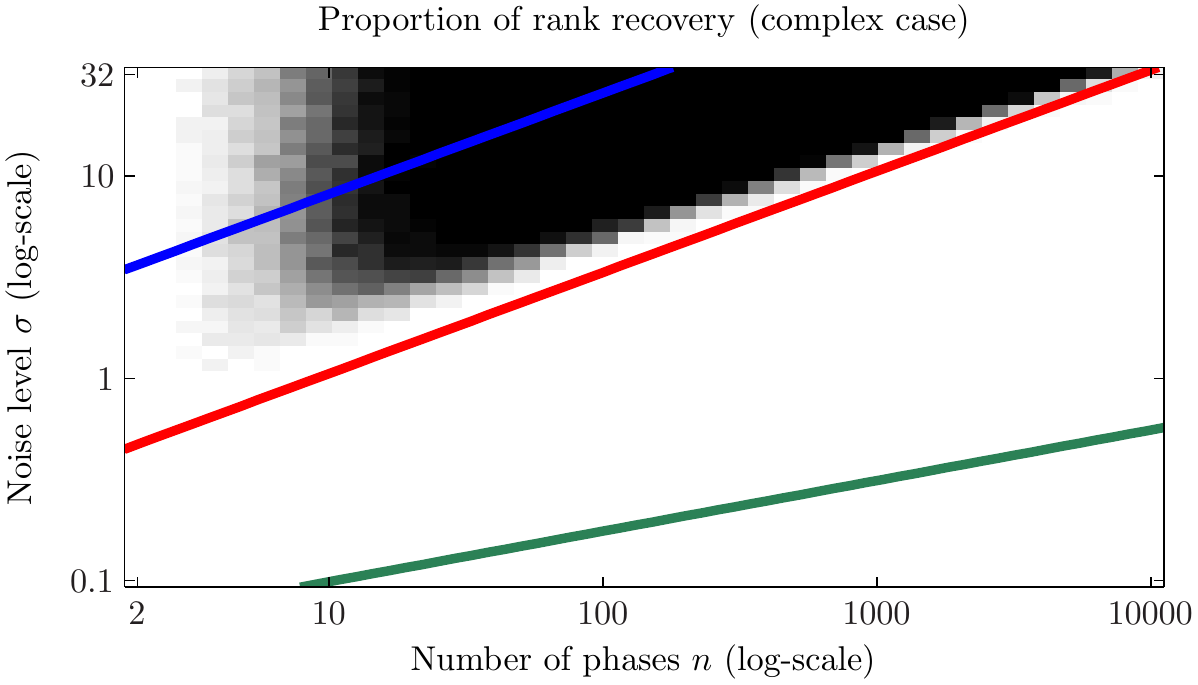}};
    \draw[draw=none,shape=rectangle] (0.1, .3)  node[mycolor1] {$\frac{1}{3}\sqrt{n}$};
    \draw[draw=none,shape=rectangle] (-5.3, 1.8)  node[mycolor2] {$\sqrt{\frac{2\pi^2}{3}}\sqrt{n}$};
    \draw[draw=none,shape=rectangle] (3.7, -1.3)  node[mycolor3] {$\frac{1}{18}n^{1/4}$};
\end{tikzpicture}
\vspace{-3mm}
\caption{In the complex case~\eqref{eq:P}, \emph{exact} recovery of the phases $z$ from the pairwise relative phase measurements $zz^* + \sigma W$ is hopeless as soon as $\sigma>0$. Nevertheless, below the \textcolor{mycolor1}{middle line}, the perturbation is smaller than the signal (in operator norm, assuming $z$-discordance), and one expects to be able to recover $z$ reasonably well.
And indeed, the maximum likelihood estimator (MLE) for $z$ is
close to $z$, as per Lemmas~\ref{lemma:Delta_L2bound} and~\ref{lem:deltainfty}. Computing the MLE is hard in general, but solving the semidefinite relaxation~\eqref{eq:SDP} is tractable. When~\eqref{eq:SDP} has a rank-one solution, that solution coincides with the MLE. This figure shows, empirically, how frequently the \eqref{eq:SDP} admits a unique rank-one solution (same color code as Figure~\ref{fig:realrecovery}). For each pair $(n, \sigma)$, 100 realizations of the noise $W$ are generated independently and \eqref{eq:SDP} is solved using a complex version of the low-rank algorithm in~\cite{boumal2015staircase}. The~\eqref{eq:SDP} appears to be tight for remarkably large levels of noise. Theorem~\ref{Thm:rr_gaussian} (our main contribution) partly explains this phenomenon, by showing that the SDP is tight with high probability below the \textcolor{mycolor3}{bottom line}.
We further note that, above the \textcolor{mycolor2}{top line}, no unbiased estimator for $z$ performs better than a random guess~\cite{BoumalManifoldSynch}.
Note that, for $n \leq 3$, a complex version of~\cite[Thm.~2.1]{pataki1998rank} guarantees deterministic tightness of the relaxation, in accordance with observation here.}
\label{fig:experiment}
\end{figure}
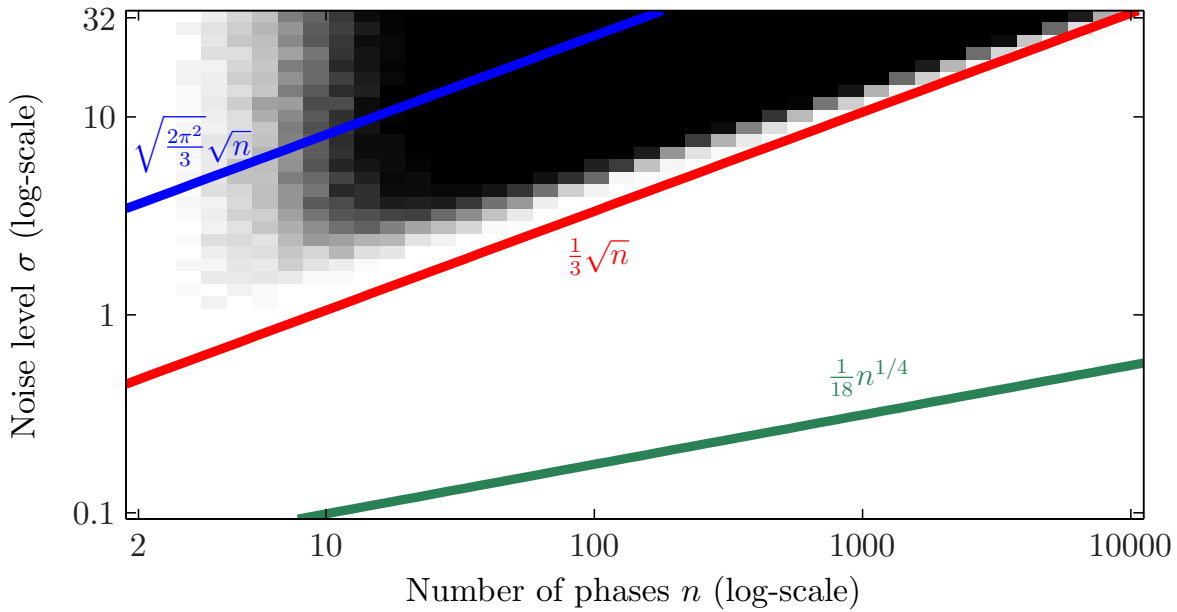

We now return to the complex case, which is the focus of this paper. As was mentioned earlier, in the presence of even the slightest noise, one can no longer reasonably expect the true signal $z$ to be an optimal solution of~\eqref{eq:P} (this can be further quantified using Cram\'er-Rao bounds~\cite{BoumalManifoldSynch}). Nevertheless, we set out to show that (under some assumptions on the noise) solutions of~\eqref{eq:P} are close to $z$ and they can be computed via~\eqref{eq:SDP}.

The proof follows that of the real case in spirit, but requires more sophisticated arguments because the solution is no longer known explicitly. This is important because the candidate dual certificate $S$ itself depends on that solution.
With this in mind, the proof of the upcoming main lemma (Lemma~\ref{lemma:main}) follows this reasoning:
\begin{enumerate}
\item For small enough noise levels $\sigma$, any optimal solution $x$ of~\eqref{eq:P} is close to the sought signal $z$ (Lemmas~\ref{lemma:Delta_L2bound} and~\ref{lem:deltainfty}).
\item Solutions $x$, a fortiori, satisfy necessary optimality conditions for~\eqref{eq:P}. First-order conditions take up the form $Sx = 0$, where $S = \Re\{\ddiag(Cx x^*)\} - C$ depends smoothly on $x$ (see~\eqref{eq:S}).
Second-order conditions will also be used.
\item Remarkably, this $S$ can be used as a dual certificate for solutions of~\eqref{eq:SDP}. Indeed, $X = xx^*$ is optimal if and only if $S$ is positive semidefinite (Lemma~\ref{lem:uniqueS}). The solution is unique if $\rank(S) = n-1$ (Lemma~\ref{lem:KKTSDP}). Thus, it only remains to study the eigenvalues of $S$.
\item In the absence of noise, $S$ is a Laplacian for a complete graph with unit weights (up to a unitary transformation), so that its eigenvalues are~$0$ with multiplicity~$1$, and~$n$ with multiplicity $n-1$. Then, $X = zz^*$ is always the unique solution.
\item Adding small noise, because of the first point, the solution $x$ will move only by a small amount, and hence so will $S$. Thus, the large eigenvalues should be controllable into remaining positive (Section~\ref{sec:controlling_large_eigs}).
\item The crucial fact follows: because of the way $S$ is constructed (using first-order optimality conditions), the zero eigenvalue is ``pinned down'' (as long as $x$ is a local optimum of \eqref{eq:P}). Indeed, both $x$ and $S$ change as a result of adding noise, but the property $Sx = 0$ remains valid. Thus, there is no risk that the zero eigenvalue from the noiseless scenario would become negative when noise is added.
\end{enumerate}
Following this road map, most of the work in the proof below consists in bounding how far away $x$ can be from $z$, and in using that to control the large eigenvalues of $S$.
This constructive way of identifying the dual certificate (third point in the roadmap) already appears explicitly in work by Journ\'ee et al.~\cite{journee2010low}, who considered a different family of real, semidefinite programs which also admit a smooth geometry when the rank is constrained. This points to smoothness of~\eqref{eq:P} (and non-degeneracy of~\eqref{eq:SDP}~\cite{alizadeh1997complementarity}) as a principal ingredient in our analysis: the KKT conditions of~\eqref{eq:P} are a subset of the KKT conditions of its relaxation. It is because the former are explicit (rather than existential as in Lemma~\ref{lem:KKTSDP}) that they help in identifying $S$.

Our main theorem follows. In a nutshell, it guarantees that: under (complex) Wigner noise $W$, with high probability, solutions of~\eqref{eq:P} are close to $z$, and, assuming the noise level $\sigma$ is smaller than (on the order of) $n^{1/4}$, \eqref{eq:SDP} admits a unique solution, it is of rank one and identifies the solution of~\eqref{eq:P} (unique, up to a global phase shift).

\begin{theorem}\label{Thm:rr_gaussian}
Let $z\in\CC^{n}$ be a vector with unit modulus entries, $W\in\CC^{n\times n}$ a Hermitian Gaussian Wigner matrix and let $C = zz^\ast + \sigma W$. Let $x \in \CC^n$ be a global optimizer of~\eqref{eq:P} satisfying $z^*x = |z^*x|$ (fixing the global phase). With probability at least $1 - \mathcal{O}(n^{-5/4})$, $x$ is close to $z$ in the following two senses:
\begin{align}
\|x-z\|_\infty & \leq 6\left( \sqrt{\log n} + 29\sigma \right) \sigma n^{-1/2}, \textrm{ and}  \nonumber\\
\|x-z\|_2 \  & \leq 12\sigma \nonumber.
\end{align}
Furthermore, if
\begin{align}\label{main_sigma_condition_2}
 \sigma \leq \frac{1}{18} n^{1/4},
\end{align}
then the semidefinite program~\eqref{eq:SDP}, given by
\[
\max_{X \in \CC^{n\times n}} \ \trace{CX}, \ \textrm{ subject to } \diag(X) = \mathds{1} \textmd{ and } X \succeq 0,
\]
has, as its unique solution, the rank-one matrix $X = xx^\ast$.
\end{theorem}
The numerical experiments (Figure~\ref{fig:experiment}) suggest it should be possible to allow $\sigma$ to grow at a rate of $\frac{n^{1/2}}{\polylog(n)}$ (as in the real case), but we were not able to establish that (see Remark~\ref{rem:Wxinfty}). Nevertheless, we do show that $\sigma$ can grow unbounded with $n$. To the best of our knowledge, this is the first result of this kind. We hope it might inspire similar results in other problems where the same phenomenon has been observed~\cite{bandeira2014open}.

%
%
%
%


\section{Main result} \label{sec:mainres}

In this section we present our main technical result and show how it can be used to prove Theorem~\ref{Thm:rr_gaussian}. We begin with a central definition in this paper. Intuitively, this definition characterizes non-adversarial noise matrices $W$.\footnote{A similar but different definition appeared in a previous version of this paper.}
\begin{definition}[$z$-discordant matrix]\label{def:discordant}
Let $z\in\CC^{n}$ be a vector with unit modulus entries. A matrix $W\in\CC^{n\times n}$ is called \emph{$z$-discordant} if it is Hermitian and satisfies both
\begin{align*}
	\opnorm{W} & \leq 3 \sqrt{n}, & \textrm{ and } & & \|Wz\|_\infty & \leq 3\sqrt{n \log n}.
\end{align*}
\end{definition}
The next lemma is the main technical contribution of this paper. It is a deterministic, non-asymptotic statement.
\begin{lemma}\label{lemma:main}
Let $z\in\CC^{n}$ be a vector with unit modulus entries, let $W\in\CC^{n\times n}$ be a Hermitian, $z$-discordant matrix (Definition~\ref{def:discordant}), and let $C = zz^\ast + \sigma W$. Let $x \in \CC^n$ be a global optimizer of~\eqref{eq:P} satisfying $z^*x = |z^*x|$. $x$ is close to $z$ in the following two senses:
\begin{align}
\|x-z\|_\infty & \leq 6\left( \sqrt{\log n} + 29\sigma \right) \sigma n^{-1/2}, \textrm{ and}  \nonumber\\
\|x-z\|_2 \  & \leq 12\sigma \nonumber.
\end{align}
Furthermore, if $\sigma \leq \frac{1}{18} n^{1/4}$,
then the semidefinite program~\eqref{eq:SDP} has, as its unique solution, the rank-one matrix $X = xx^\ast$.

\end{lemma}

We defer the proof of Lemma~\ref{lemma:main} to Section~\ref{section:proof}. The following proposition, whose proof we defer to Appendix~\ref{apdx:WignerDiscordant}, shows how this lemma can be used to prove Theorem~\ref{Thm:rr_gaussian}.

\begin{proposition}\label{prop:WignerDiscordant}
Let $z\in\CC^n$ be a (deterministic) vector with unit modulus entries. Let $W \in \CC^{n\times n}$ be a Hermitian, random matrix
with i.i.d.\ off-diagonal entries following a standard complex normal distribution and zeros on the diagonal.
Thus, $W_{ii} = 0$, $W_{ij} = \overline{W_{ji}}$, $\EE W_{ij} = 0$ and $\EE |W_{ij}|^2 = 1$ (for $i \neq j$). Then, $W$ is $z$-discordant with probability at least $1 - 2n^{-5/4} - e^{-n/2}$.
\end{proposition}

The latter result is not surprising. Indeed, the definition of $z$-dis\-cor\-dance requires two elements. Namely, (i) that $W$ be not too large as an operator, and (ii) that no row of $W$ be too aligned with $z$. For $W$ a Wigner matrix independent of $z$, those are indeed expected to hold. In fact, for Gaussian noise, the constants 3 can be replaced by $2+\varepsilon$ for any $\varepsilon > 0$, provided $n$ is large enough.

The definition of $z$-discordance is not tightly adjusted to Wigner noise. As a result, it is expected that Lemma~\ref{lemma:main} will be applicable to show tightness of semidefinite relaxations for a larger span of noise models.


\section{The proof}\label{section:proof}

In this section, we prove Lemma~\ref{lemma:main}. See Section~\ref{sec:synchroSO2roadmap} for an outline of the proof.

\subsection{Global optimizers of~\eqref{eq:P} are close to $z$}

This subsection focuses on bounding the maximum likelihood estimation error, in $\ell_2$ and $\ell_\infty$ norms. Later on, only the $\ell_2$ bound is used to establish tightness of the SDP.

\begin{lemma}\label{lemma:Delta_L2bound}
If $W$ is $z$-discordant and $x\in\mathbb{C}^n$ verifies $\|x\|_2^2 = n$ and $z^*Cz \leq x^*Cx$ (in particular, if $x$ is a global optimizer of~\eqref{eq:P}), then
\begin{align*}
	\min_{\theta \in \mathbb{R}}\|xe^{i \theta}-z\|_2^2 & = 2(n - |z^*x|) \leq 144 \sigma^2.
\end{align*}
\end{lemma}
\begin{proof}
The $\ell_2$ error and the correlation between $z$ and $x$ are related as follows:
\begin{align}
	\min_{\theta \in \mathbb{R}}\|xe^{i \theta}-z\|_2^2 & = 2\left(n - \max_{\theta\in\mathbb{R}} \Re\{e^{i\theta} z^*x\} \right) = 2(n - |z^*x|).
	\label{eq:l2correlation}
\end{align}
Without loss of generality, assume the global phase of $x$ is such that $z^*x = |z^*x|$, i.e., $x$ and $z$ are optimally aligned.
Expand the inequality $z^* C z \leq x^* C x$ using the data model $C = zz^* + \sigma W$ to obtain
$$
	n^2 + \sigma z^*Wz \leq |z^*x|^2 + \sigma x^* W x.
$$
Group noise terms as follows:
\begin{align*}
	n^2 - |z^*x|^2 \leq \sigma \left( x^* W x - z^* W z \right).
\end{align*}
Since $n^2 - |z^*x|^2 = (n - |z^*x|)(n + |z^*x|)$, divide both sides by $n + |z^*x| \geq n$ to obtain
\begin{align}
	n - |z^*x| \leq \sigma \left( x^* W x - z^* W z \right)n^{-1}.
	\label{eq:l2error_intermediate}
\end{align}
The difference of quadratic terms should be small if $x$ and $z$ are close. Indeed, using $\opnorm{W} \leq 3n^{1/2}$ (by $z$-discordance) and $\|x\|_2 = \|z\|_2 = n^{1/2}$:
\begin{align*}
	x^* W x - z^* W z & = \Re\{(x-z)^*W(x+z)\}\\
					  & \leq \|x-z\|_2 \cdot \opnorm{W} \cdot \|x+z\|_2 \\
					  & \leq 6n \|x-z\|_2.
\end{align*}
Combine this and~\eqref{eq:l2correlation} with~\eqref{eq:l2error_intermediate} to obtain $\|x-z\|_2 \leq 12\sigma$. %
\end{proof}
Note that the constant 12 is pessimistic, because we used $n+|z^*x| \geq n$ even though it is closer to $2n$. Assuming $\|x-z\|_2 \leq \alpha \sigma$ and $\sigma \leq t\sqrt{n}$, the argument above can be bootstrapped to reduce the constant. One obtains $\|x-z\|_2 \leq \alpha_k \sigma$ for all $k$, with $\alpha_0 = 12$ and $\alpha_{k+1} = 6 + (t/2)^2\alpha_k^3$. For $t < 1/6\sqrt{2}$ small enough, $\alpha_k$ converges arbitrarily close to 6. For example, if $\sigma \leq \sqrt{n}/12$, then $\|x-z\|_2 \leq 6.5\sigma$.

The next lemma establishes a bound on the largest individual error, $\|x-z\|_\infty$, after proper global phase alignment of $x$ and $z$.
Interestingly, for $\sigma \ll n^{1/4}$, the bound shows that individual errors decay, uniformly, as $n$ increases.

\begin{lemma} \label{lem:deltainfty}
If $W$ is $z$-discordant and $x$ is a global optimizer of~\eqref{eq:P} with global phase such that $z^*x = |z^*x|$ (that is, the phase that minimizes the $\ell_2$ distance), then
\begin{align*}
	\|x-z\|_\infty & \leq 6\left( \sqrt{\log n} + 29\sigma \right) \sigma n^{-1/2}. %
\end{align*}
\end{lemma}
\begin{proof}
We wish to upper bound, for all $i\in \{1, 2, \ldots, n\}$, the value of $|x_i - z_i|$.
Let $e_i \in \RR^n$ denote the $i$th vector of the canonical basis (its $i$th entry is 1 whereas all other entries are zero). Consider $\hat x = x + (z_i-x_i)e_i$, a feasible point of~\eqref{eq:P} obtained from the optimal $x$ by changing its $i$th entry to the corresponding entry of $z$. Since $x$ is optimal, it performs at least as well as $\hat x$ according to the cost function of~\eqref{eq:P}:
\begin{align}
	x^*Cx \geq \hat x^* C \hat x = x^*Cx + |z_i-x_i|^2 C_{ii} + 2\inner{z_i-x_i}{(Cx)_i},
	\label{eq:ineqlinfty}
\end{align}
with the (real) inner product~\eqref{eq:inner} over $\mathbb{C}$. Since $C_{ii} = 1 \geq 0$,\footnote{The inequality is independent of $C_{ii}$. Assuming  nonnegativity merely eases the exposition.} it follows that the right-most term is nonpositive. Using $C = zz^* + \sigma W$, we have
\begin{align*}
	\inner{z_i-x_i}{(z^*x) z_i + \sigma (Wx)_i} \leq 0.
\end{align*}
Observe that
\begin{align}
	|x_i-z_i|^2 = 2(1 - \inner{x_i}{z_i})
	 = 2\inner{z_i - x_i}{z_i},
	\label{eq:observation}
\end{align}
and combine with the assumption $z^*x = |z^*x|$ to obtain
\begin{align}
	|z^*x| |x_i - z_i|^2  \leq 2\sigma \inner{x_i-z_i}{ (Wx)_i} \leq 2\sigma |x_i - z_i| |(Wx)_i|.
	\label{eq:inftyforone}
\end{align}
This holds for all $i$, hence
\begin{align*}
	|z^*x| \|x-z\|_\infty \leq 2\sigma \|Wx\|_\infty.
\end{align*}
Invoke Lemma~\ref{lemma:Delta_L2bound}, namely, $|z^*x| \geq n - 72\sigma^2$,
to get
\begin{align*}
	n\|x-z\|_\infty \leq 2\sigma \|Wx\|_\infty + 72\sigma^2 \|x-z\|_\infty.
\end{align*}
Since $\|x-z\|_\infty \leq 2$, it finally comes that
\begin{align*}
	\|x - z\|_\infty & \leq 2 \left( \|Wx\|_\infty + 72\sigma \right) \sigma n^{-1}.
\end{align*}
We discuss the problem of bounding $\|Wx\|_\infty$ in more details later on. For now, we simply use the suboptimal bound~\eqref{eq:Wxinftybound} (obtained independently of the present lemma), i.e., $\|Wx\|_\infty \leq 3\sqrt{n\log n} + 36\sigma \sqrt{n}$. Then, for all $n \geq 2$, using $24/\sqrt{2} \leq 17$,
\begin{align*}
\|x - z\|_\infty & \leq 6 \left( \sqrt{\log n} + 12\sigma + 24\frac{\sigma}{\sqrt{n}} \right) \sigma n^{-1/2} \leq 6 \left( \sqrt{\log n} + 29\sigma \right) \sigma n^{-1/2}.
\end{align*}
(For $\sigma \ll n^{1/2}$, the constant $29$ could be replaced by one arbitrarily close to $12$.)
\end{proof}

\subsection{Optimality conditions for~\eqref{eq:SDP}}

The global optimizers of the semidefinite program~\eqref{eq:SDP} can be characterized completely via the Karush-Kuhn-Tucker (KKT) conditions.
\begin{lemma}\label{lem:KKTSDP}
A Hermitian matrix $X \in \CC^{n\times n}$ is a global optimizer of~\eqref{eq:SDP} if and only if there exists a Hermitian matrix $\hat S\in\CC^{n\times n}$
such that all of the following hold:
\begin{enumerate}
\item $\diag(X) = \mathds{1}$;
\item $X \succeq 0$;
\item $\hat SX = 0$;
\item $\hat S+C$ is (real) diagonal; and %
\item $\hat S \succeq 0$.
\end{enumerate}
If, furthermore, $\rank(\hat S) = n-1$, then $X$ has rank one and is the unique global optimizer of~\eqref{eq:SDP}.
\end{lemma}
\begin{proof}
These are the KKT conditions of~\eqref{eq:SDP}~\cite[Example 3.36]{ruszczynski2006nonlinear}. Conditions 1 and 2 are primal feasibility, condition 3 is complementary slackness and conditions 4 and 5 encode dual feasibility. Since the identity matrix $I$ is strictly feasible, the \emph{Slater condition} is fulfilled. This ensures the KKT conditions are necessary and sufficient for global optimality~\cite[Theorem 3.34]{ruszczynski2006nonlinear}. Slater's condition also holds for the dual. Indeed, let $\tilde S = \alpha I - C$, where $\alpha \in \RR$ is such that $\tilde S \succ 0$ (such an $\alpha$ always exists); then $\tilde S + C$ is indeed diagonal and $\tilde S$ is strictly admissible for the dual. This allows to use results from~\cite{alizadeh1997complementarity}. Specifically, assuming $\rank(\hat S) = n-1$, Theorem 9 in~\cite{alizadeh1997complementarity} implies that $\hat S$ is \emph{dual nondegenerate}. Then, since $\hat S$ is also optimal for the dual (by complementary slackness), Theorem 10 in~\cite{alizadeh1997complementarity} guarantees that the primal solution $X$ is unique. Since $X$ is nonzero and $\hat S X = 0$, it must be that $\rank(X) = 1$.
\end{proof}

Certainly, if~\eqref{eq:SDP} admits a rank-one solution, it has to be of the form $X = xx^*$, with $x$ an optimal solution of the original problem~\eqref{eq:P}. Based on this consideration, our proof of Lemma~\ref{lemma:main} goes as follows. We let $x$ denote a global optimizer of~\eqref{eq:P} and we consider $X = xx^*$ as a candidate solution for~\eqref{eq:SDP}. Using the optimality of $x$ and assumptions on the noise, we then construct and verify a dual certificate matrix~$\hat S$ as required per Lemma~\ref{lem:KKTSDP}. In such proofs, one of the nontrivial parts is to guess an analytical form for $\hat S$ given a candidate solution $X$. We achieve this by inspecting the first-order optimality conditions of~\eqref{eq:P} (which $x$ necessarily satisfies). The main difficulty is then to show the suitability of the candidate $S$, as it depends nonlinearly on the global optimum $x$, which itself is a complicated function of the noise $W$. We show feasibility of $S$ via a program of inequalities, relying on $z$-discordance of the noise~$W$ (see Definition~\ref{def:discordant}).

\subsection{Construction of the dual certificate $S$}

Every global optimizer of the combinatorial problem~\eqref{eq:P} must, a fortiori, satisfy first-order necessary optimality conditions. We derive those now.

Endow the complex plane $\CC$ with the Euclidean metric
\begin{align}
	\inner{y_1}{y_2} = \Re\{ y_1^* y_2^{} \}.
	\label{eq:inner}
\end{align}
This is equivalent to viewing $\CC$ as $\RR^2$ with the canonical inner product, using the real and imaginary parts of a complex number as its first and second coordinates. Denote the complex circle by
$$\mathcal{S} = \{ y \in \CC : y^*y = 1 \}.$$
The circle can be seen as a submanifold of $\CC$, with tangent space at each $y$ given by (simply differentiating the constraint):
$$T_y\mathcal{S} = \{ \dot y \in \CC : \dot y^* y + y^* \dot y = 0 \} = \{ \dot y \in \CC : \inner{y}{\dot y} = 0 \}.$$
Restricting the Euclidean inner product to each tangent space equips $\mathcal{S}$ with a Riemannian submanifold geometry. The search space of~\eqref{eq:P} is exactly $\mathcal{S}^n$, itself a Riemannian submanifold of $\CC^n$ with the product geometry. Thus, problem~\eqref{eq:P} consists in \emph{minimizing} a smooth function
\begin{align*}
	g(x) & = -x^*Cx
\end{align*}
over the smooth Riemannian manifold $\mathcal{S}^n$. Therefore, the first-order necessary optimality conditions for~\eqref{eq:P} (i.e., the KKT conditions) can be stated simply as $\operatorname{grad} g(x) = 0$, where $\operatorname{grad} g(x)$ is the Riemannian gradient of $g$ at $x \in \mathcal{S}^n$~\cite{AMS08}. This gradient is given by the orthogonal projection
of the Euclidean (the classical) gradient of $g$ onto the tangent space of $\mathcal{S}^n$ at $x$~\cite[eq.\,(3.37)]{AMS08},
\begin{align}
	T_x\mathcal{S}^n & = T_{x_1}\mathcal{S} \times \cdots \times \nonumber T_{x_n}\mathcal{S}\\
		& = \left\{ \dot x \in \mathbb{C}^n : \Re\{\diag(\dot x x^*)\} = 0 \right\}.
	\label{eq:TxSn}
\end{align}
The projector and the Euclidean gradient are given respectively by:
\begin{align}
	 \operatorname{Proj}_x & \colon \CC^n \to T_x\mathcal{S}^n \colon \dot x \mapsto \operatorname{Proj}_x \dot x = \dot x - \Re\{\ddiag(\dot x x^*)\}x, \nonumber\\
	\nabla g(x) & = -2Cx, \nonumber
\end{align}
where $\ddiag \colon \CC^{n\times n} \to \CC^{n\times n}$ sets all off-diagonal entries of a matrix to zero.
Thus, for $x$ a global optimizer of~\eqref{eq:P}, it holds that
\begin{align}
	0 = \operatorname{grad} g(x)  = \operatorname{Proj}_x \nabla g(x) %
	& = 2(\Re\{\ddiag(Cx x^*)\} - C)x.
	\label{eq:gradg}
\end{align}
This suggests the following definitions:
\begin{align}
	X & = xx^*, & S & = \Re\{\ddiag(Cx x^*)\} - C.
	\label{eq:S}
\end{align}
Note that $S$ is Hermitian and $Sx = 0$. Referring to the KKT conditions in Lemma~\ref{lem:KKTSDP}, it follows immediately that $X$ is feasible for~\eqref{eq:SDP} (conditions 1 and 2); that $SX = (Sx)x^* = 0$ (condition 3); and that $S+C$ is a diagonal matrix (condition 4). It thus only remains to show that $S$ is also positive semidefinite and has rank $n-1$. If such is the case, then $X$ is the unique global optimizer of~\eqref{eq:SDP}.
Note the special role of the first-order necessary optimality conditions: they guarantee complementary slackness, without requiring further work.

We will also use second-order necessary optimality conditions, namely, that the (Riemannian) Hessian of $g$ at an optimizer $x$ is positive semidefinite on the tangent space~\eqref{eq:TxSn}. The action of the Riemannian Hessian is obtained by projecting the directional derivatives of the Riemannian gradient~\cite[eq.\,(5.15)]{AMS08}. Explicitly, for any tangent vector $\dot x \in T_x\mathcal{S}^n$, one can compute that (with $\operatorname{D}\!h(x)[\dot x]$ denoting the directional derivative of $h$ at $x$ along $\dot x$)
\begin{align*}
	\operatorname{Hess} g(x)[\dot x] = \operatorname{Proj}_x \operatorname{D} \operatorname{grad} g(x)[\dot x] = \operatorname{Proj}_x 2S\dot x.
\end{align*}
Thus, if $x$ is a (local) optimizer, then, since $\operatorname{Proj}_x$ is self-adjoint,
\begin{align*}
	\inner{\dot x}{\operatorname{Hess} g(x)[\dot x]} = 2\inner{\dot x}{S \dot x} \geq 0
\end{align*}
for all $\dot x \in T_x\mathcal{S}^n$: a necessary but insufficient step towards making $S$ positive semidefinite. We will later use this condition along selected directions.

The following lemma further shows that $S$ is the right candidate dual certificate. More precisely, for $x$ a critical point of~\eqref{eq:P}, it is necessary and sufficient for $S$ to be positive semidefinite in order for $X = xx^*$ to be optimal for~\eqref{eq:SDP}. This confirms that, by studying $S$, nothing is lost with respect to the original question. See also~\cite{boumal2015staircase}.
\begin{lemma}\label{lem:uniqueS}
A feasible $X$ (of any rank) for~\eqref{eq:SDP} is optimal if and only if $S = \Re\{\ddiag(CX)\} - C$~\eqref{eq:S} is positive semidefinite. There exists no other certificate.
\end{lemma}
\begin{proof}
The \emph{if} part follows from Lemma~\ref{lem:KKTSDP}: set $\hat S = S$ and observe that, since $\trace{SX} = 0$ by construction, and since $S,X\succeq 0$, it follows that $SX = 0$. We now show the \emph{only if} part. Assume $X$ is optimal. Then, by Lemma~\ref{lem:KKTSDP}, there exists $\hat S \succeq 0$ which satisfies $\hat S X = 0$ and $\hat S + C = \hat D$, where $\hat D$ is diagonal. Thus, $CX = (\hat D - \hat S)X = \hat D X$ and $\Re\{\ddiag(CX)\} = \hat D$. Consequently, $S = \hat D - C = \hat S$.
\end{proof}

\subsection{A sufficient condition for rank recovery} %
\label{sec:controlling_large_eigs}

Knowing which certificate $S$~\eqref{eq:S} to verify for a given $x$, it remains to characterize the point $x$. Of course, $x$ is the global optimizer of~\eqref{eq:P}, but this is not a convenient property to exploit. Instead, we let $x$ be a \emph{second-order critical point}\footnote{A second-order critical point satisfies first- and second-order necessary optimality conditions, namely, the gradient is zero and the Hessian is positive semidefinite (if minimizing)~\cite[\S3.2.1]{conn2000trust}.} for~\eqref{eq:P} which outperforms the planted signal $z$ (certainly, the MLE is one such point). In effect, we prove the following, which immediately yields Lemma~\ref{lemma:main}.
\begin{proposition}
	Let $z\in\CC^n$ have unit-modulus entries, $W$ be $z$-discordant, $\sigma \leq \frac{1}{18}n^{1/4}$ and $C = zz^*+\sigma W$. If $x\in\CC^n$ is a second-order critical point for~\eqref{eq:P} such that $x^*Cx \geq z^*Cz$, then $xx^*$ is the unique solution of~\eqref{eq:SDP} and $x$ solves~\eqref{eq:P}.
\end{proposition}
To this end, we prove that, at such points $x$, the certificate $S$ is positive semidefinite of rank $n-1$.
As a first observation, note that $x$ being critical ($Sx = 0$) implies that $\Re\{(Cx)_i \bar x_i\}x_i = (Cx)_i$ for all $i$. Thus, $(Cx)_i\bar x_i$ is real, and it follows that 
$$
	S = \ddiag(Cx x^*) - C.
$$
Furthermore, since $x$ is second-order critical, the Hessian of $g$ at $x$ is positive semidefinite, implying $\inner{\dot x}{S \dot x} \geq 0$ for all $\dot x \in T_x\mathcal{S}^n$~\eqref{eq:TxSn}. In particular, for all $i$, $\dot x = (jx_i)e_i$ is a tangent vector at $x$ (where $j = \sqrt{-1}$ and $e_i$ is the $i$th canonical basis vector), and it follows that 
$$
	\inner{(jx_i)e_i}{(jx_i)Se_i} = S_{ii} = (Cx)_i\bar x_i - C_{ii} \geq 0.
$$
Since $C_{ii} = 1$ by construction,\footnote{As before, $S$ is independent of $\operatorname{diag}(C)$. Assuming $C_{ii} = 1$ involves no loss of generality.} it follows in particular that $(Cx)_i\bar x_i$ is real and positive. Thus:
\begin{align*}
	(Cx)_i\bar x_i = |(Cx)_i\bar x_i| = |(Cx)_i|.
\end{align*}

We can now turn to lower-bounding the eigenvalues of $S$. Since $Sx = 0$, it is sufficient to study $u^*Su$ for $u\in\mathbb{C}^n$ orthogonal to $x$, that is, $u^*x = 0$. Without loss of generality, fix the global phase of $x$ such that $z^*x = |z^*x|$. Using once more that $C = zz^* + \sigma W$,
\begin{align*}
	u^*Su & = u^* \operatorname{ddiag}(Cxx^*)u - u^*Cu \\
		  & = \sum_{i=1}^{n} |u_i|^2 |(Cx)_i|  - u^*Cu \\
		  & = \sum_{i=1}^{n} |u_i|^2 |(z^*x)z_i + \sigma (Wx)_i|  - |u^*z|^2 - \sigma u^* W u \\
		  & \geq |z^*x|\|u\|_2^2 - \sigma \sum_{i=1}^{n} |u_i|^2 |(Wx)_i|  - |u^*(z-x)|^2 - \sigma \opnorm{W}\|u\|_2^2 \\
		  & \geq \|u\|_2^2 \left( |z^*x| - \sigma \frac{\sum_{i=1}^{n} |u_i|^2 |(Wx)_i|}{\sum_{i=1}^{n} |u_i|^2} - \|z-x\|_2^2 - \sigma \opnorm{W} \right).
\end{align*}
The fraction is a weighted mean of nonnegative numbers, $|(Wx)_i|$ for $i = 1\ldots n$, so that it is bounded by $\|Wx\|_\infty$. Now using $z$-discordance of $W$, invoke Lemma~\ref{lemma:Delta_L2bound} to bound $|z^*x| \geq n - 72\sigma^2$, $\|z-x\|_2^2 \leq 144\sigma^2$, and $\opnorm{W} \leq 3n^{1/2}$ to establish
\begin{align*}
	u^* S u \geq \|u\|_2^2 \big( n - 216\sigma^2 - \sigma \|Wx\|_\infty - 3\sigma n^{1/2} \big).
\end{align*}
As a result, a sufficient condition for $S$ to be positive semidefinite with rank $n-1$ is to have
\begin{align}
	n > \sigma \left( 216\sigma + 3 n^{1/2} + \|Wx\|_\infty \right).
	\label{eq:sufficient1}
\end{align}
At this point, notice that if $\|Wx\|_\infty$ were bounded by $\tilde{\mathcal{O}}(\sqrt{n})$, we would indeed obtain the targeted rate for $\sigma$. Unfortunately, we could not obtain such a bound (see Remark~\ref{rem:Wxinfty}). Instead, we carry on with the following suboptimal argument:
\begin{align*}
	\|Wx\|_\infty %
				  & \leq \|W(x - z)\|_\infty + \|Wz\|_\infty \\
				  & \leq \|W(x - z)\|_2 + \|Wz\|_\infty \\
				  & \leq \opnorm{W} \|x - z\|_2 + \|Wz\|_\infty.
\end{align*}
By $z$-discordance, $\opnorm{W} \leq 3 \sqrt{n}$ and $\|Wz\|_\infty \leq 3\sqrt{n \log n}$. By Lemma~\ref{lemma:Delta_L2bound}, $\|x - z\|_2 \leq 12\sigma$. Thus,
\begin{align}
	\|Wx\|_\infty \leq 36\sigma \sqrt{n} + 3\sqrt{n\log n}.
	\label{eq:Wxinftybound}
\end{align}
Combining with~\eqref{eq:sufficient1} yields the following sufficient condition:
\begin{align}
	\sqrt{n} > 3\sigma \left( 72\sigma n^{-1/2} + 1 + 12\sigma + \sqrt{\log n} \right).
	\label{eq:bestcondition}
\end{align}
This condition involves only $\sigma$ and $n$, and as such is the best result we prove.
The bottleneck is the term in $\sigma^2$, which leads to a bound of the type $\sigma \leq c \cdot n^{1/4}$.
Further inspection of~\eqref{eq:bestcondition} shows $\sigma \leq \frac{1}{18} n^{1/4}$ is indeed a sufficient condition, for $n\geq 2$. This concludes the proof of Lemma~\ref{lemma:main}.

\begin{remark}[About $\|Wx\|_\infty$]\label{rem:Wxinfty}
Each entry $(Wx)_i$ is a sum of $n-1$ random variables in the complex plane. If $W_{ij}$ are standard (complex) Gaussians (so that their phases are uniformly distributed) and if $x$ were independent from $W$, then the quantity $|(Wx)_i|$ would grow as ${\tilde{\mathcal{O}}}(n^{1/2})$. The largest one, $\|Wx\|_\infty$, would grow similarly (as is the case for $\|Wz\|_\infty$, where we do have independence). Thus, we would recover the observed rate for $\sigma$ in Figure~\ref{fig:experiment}, namely, $\sigma \leq \frac{n^{1/2}}{\operatorname{polylog}(n)}$ would suffice to ensure tightness. Likewise, the proof in Lemma~\ref{lem:deltainfty} would lead to $\|x-z\|_\infty = \tilde{\mathcal{O}}(\sigma / \sqrt{n})$. Unfortunately, $W$ and $x$ are not independent: the MLE depends on $W$ in a complicated fashion.

One avenue to improve the current bound might be to expand $x = x(\sigma W)$ around~0, to first order. Some computations (omitted here) show that such a development would lead to $x = z + \frac{\sigma}{n} \Im\left\{ \ddiag(Wzz^*) \right\}(iz) + \epsilon$. Then,
\begin{align*}
	\|Wx\|_\infty \lesssim \|Wz\|_\infty + \frac{\sigma}{n}\|W\Im\left\{ \ddiag(Wzz^*) \right\}z\|_\infty.
\end{align*}
Still assuming Gaussian entries, the last norm appears
to grow like $\tilde{\mathcal{O}}(n)$, which would lead to $\|Wx\|_\infty = \tilde{\mathcal{O}}(\sqrt{n}+\sigma)$, as desired. Unfortunately, controlling the error term $\epsilon$ and its effect on the above bound seems nontrivial, as this involves assessing smoothness of the argmax function $x(\sigma W)$.
\end{remark}

\section{Conclusions} \label{sec:conclusions}

Computing maximum likelihood estimators is important for statistical problems, in particular when exact signal recovery is out of reach. In this paper, we have made partial progress toward understanding when this computation can be executed in polynomial time for the angular synchronization problem. It still remains to improve the proposed result to understand how this MLE can be computed for noise levels as high as those witnessed empirically. As it stands, the suboptimal bound~\eqref{eq:Wxinftybound} is the bottleneck. Reducing it to $\tilde{\mathcal{O}}(\sqrt{n})$ would yield the (tentatively) correct rate, as per Remark~\ref{rem:Wxinfty}.

For many other problems of interest, it has likewise been observed that semidefinite relaxations tend to be tight. Notably, this is the case for synchronization of rotations: a problem similar to angular synchronization but with rotations in $\mathbb{R}^d$ rather than~$\mathbb{R}^2$. Analyzing the latter might prove more difficult, because rotations in $\mathbb{R}^d$ do not commute for $d > 2$. Even for $d=2$, the case of an incomplete graph of (weighted) measurements is of interest, and not covered here. Tightness has also been observed for the Procrustes problem and 
the multi-reference alignment (MRA) problem~\cite{Bandeira_Charikar_Singer_Zhu_Alignment}. The situation for MRA differs markedly from the present setting, as
the noise originates at the vertices of a graph, not on the edges.
Notwithstanding those differences, we hope that the present work will help future attempts to explain the tightness phenomenon in various settings.

\appendix

\section{Wigner matrices are discordant}
\label{apdx:WignerDiscordant}

This appendix is a proof for Proposition~\ref{prop:WignerDiscordant}, namely, that for arbitrary $z\in\CC^n$ such that $|z_1| = \cdots = |z_n| = 1$, complex Wigner matrices are $z$-discordant (Definition~\ref{def:discordant}) with high probability.

A matrix $W$ is $z$-discordant if and only if $\diag(z)^* W \diag(z)$ is $\mathds{1}$-discordant. Since $\diag(z)^* W \diag(z)$ has the same distribution as $W$ (owing to complex normal random variables having uniformly random phase), we may without loss of generality assume $z = \mathds{1}$ in the remainder of the proof.

\begin{enumerate}

\item $\Pr\left\{ \opnorm{W} > 3 n^{1/2} \right\} \leq e^{-n/2} $.

Although tail bounds for the real version of this are well-known (see for example~\cite{VershyninNARandomMatrices,bandeira2014sharp}) and they mostly hold verbatim in the complex case, for the sake of completeness we include a classical argument, based on Slepian's comparison theorem and Gaussian concentration, for a tail bound in the complex valued case.

We will bound the largest eigenvalue of $W$. It is clear that a simple union bound argument will allow us 
to bound also the smallest, and thus bound the largest in magnitude. Let $\lambda_+ = \max_{v\in\CC^n: \|v\|=1}v^\ast W v$ denote the largest eigenvalue of $W$. For any unit-norm $u,v\in\CC^n$, the real valued Gaussian process $X_v =  v^\ast Wv$ satisfies:
\begin{align*}
\EE\left(X_v - X_u\right)^2 & = \EE\left( \sum_{i<j}W_{ij}\left(\overline{v_i}v_j -\overline{u_i}u_j\right)   + W_{ji}\left(\overline{v_j}v_i -\overline{u_j}u_i\right)\right)^2 \\
& = \sum_{i<j}\EE \left[W_{ij}\left(\overline{v_i}v_j -\overline{u_i}u_j\right)   + W_{ji}\left(\overline{v_j}v_i -\overline{u_j}u_i\right)\right]^2.
\end{align*}
The variable $W_{ij}$ has uniformly random phase, hence so does $W_{ij}^2$, so that $\EE W_{ij}^2 = 0$. As a result,
\begin{align*}
\EE\left(X_v - X_u\right)^2 & =  \sum_{i<j}2\EE\left|W_{ij}\right|^2\left|\overline{v_i}v_j -\overline{u_i}u_j\right|^2 \\
& = 2\sum_{i<j}\left|\overline{v_i}v_j -\overline{u_i}u_j\right|^2 \\
& \leq \sum_{i,j}\left|\overline{v_i}v_j -\overline{u_i}u_j\right|^2.
\end{align*}

Note that, since $\|u\|=\|v\|=1$,
\begin{align*}
\sum_{i,j}\left|\overline{v_i}v_j -\overline{u_i}u_j\right|^2 & = \sum_{ij}\left[ |v_i|^2|v_j|^2 + |u_i|^2|u_j|^2 - \overline{v_i}v_ju_i\overline{u_j} - v_i\overline{v_j}\overline{u_i}u_j\right] \\
& = 2- 2\left|v^\ast u\right|^2 \\
& \leq 2\left( 2 -  2\left|v^\ast u\right|\right) \\
& \leq 4\left( 1 -  \Re\left[v^\ast u \right]\right)\\
 & = 2\| v-u \|^2.
\end{align*}

This means that we can use Slepian's comparison theorem (see for example~\cite[Cor.\,3.12]{ledoux1991probability}) to get
\begin{align}\label{Slepians_lemma}
\EE \lambda_+ \leq \sqrt{2}\EE \max_{\tilde{v}\in\RR^{2n}:\|\tilde{v}\|=1} \tilde{v}^T g \leq 2\sqrt{n},
\end{align}
where $g$ is a standard Gaussian vector in $\RR^{2n}$ and $\tilde{v}$ is a vector in $\RR^{2n}$ obtained from $v\in \CC^{n}$ by stacking its real and imaginary parts.

Since $\left| \|W_1\| - \|W_2\| \right| \leq  \|W_1-W_2\| \leq \|W_1-W_2\|_{\mathrm{F}}$, Gaussian concentration~\cite{ledoux1991probability} gives
\begin{align}\label{eq:gaussian_concentration}
\Pr\left\{ \lambda_+ - \EE \lambda_+ \geq t \right\} \leq e^{-t^2/2}.
\end{align}

Using (\ref{Slepians_lemma}) and (\ref{eq:gaussian_concentration}) gives
\[
\Pr\left\{ \opnorm{W} > 3 n^{1/2} \right\} \leq e^{-n/2}.
\]

\item $\Pr\left\{ \|W\mathds{1}\|_\infty > 3\sqrt{n \log n} \right\} \leq 2n^{-5/4}$.

The random vector given by $\frac1{(n-1)^{1/2}}W\mathds{1}$ is jointly Gaussian where the marginal of each entry is a standard complex Gaussian. By a suboptimal union bound argument, %
the maximum absolute value among $k$ standard complex Gaussian random variables (not necessarily independent) is larger than $t$ with probability at most $2ke^{-t^2/4}$. Hence,
\begin{align*}
\Pr\left\{ \|W\mathds{1}\|_\infty > 3\sqrt{n\log n} \right\} & \leq \Pr\left\{ \left\|\frac1{(n-1)^{1/2}}W\mathds{1}\right\|_\infty > 3 \sqrt{\log n} \right\} \\
& \leq 2ne^{-\frac{9}{4} \log n} = 2n^{-5/4}.
\end{align*}

\end{enumerate}

To support the discussion following Proposition~\ref{prop:WignerDiscordant}, we further argue that $$\Pr\left\{  |\mathds{1}^* W \mathds{1}| > n^{3/2} \right\} \leq e^{-n/2}.$$
It is easy to see that $\mathds{1}^* W \mathds{1}$ is a real Gaussian random variable with zero mean and variance $2\frac{n(n-1)}2 = n(n-1)$. This implies that:
\begin{align*}
\Pr\left\{ |\mathds{1}^* W \mathds{1}| > n^{3/2} \right\} & \leq \Pr\left\{ \frac1{(n(n-1))^{1/2}} |\mathds{1}^* W \mathds{1}| > n^{1/2} \right\} \\
&\leq \frac1{\sqrt{2\pi}}\frac1{n^{1/2}}e^{-\frac{n}2} \leq e^{-n/2}.
\end{align*}

\section*{Acknowledgments}
A.\ S.\ Bandeira was supported by AFOSR Grant No.\ FA9550-12-1-0317. Most of this work was done while he was with the Program for Applied and Computational Mathematics at Princeton University.
N.\ Boumal was supported by a Belgian F.R.S.-FNRS fellowship while working at the Universit\'e catholique de Louvain (Belgium), by a Research in Paris fellowship at Inria and ENS, the ``Fonds Sp\'eciaux de Recherche'' (FSR UCLouvain), the Chaire Havas ``Chaire Eco\-no\-mie et gestion des nouvelles don\-n\'ees'' and the ERC Starting Grant SIPA.
A.\ Singer was partially supported by Award Number R01GM090200 from the NIGMS, by Award Numbers FA9550-12-1-0317 and FA9550-13-1-0076 from AFOSR, by
Award Number LTR DTD 06-05-2012 from the Simons Foundation, and by the Moore Foundation.

\bibliographystyle{plain}
\bibliography{rankrecovery}

\end{document}